\documentclass[12pt, reqno]{amsart}
\usepackage{amsmath, amsthm, amscd, amsfonts, amssymb, graphicx, color, mathrsfs}
\usepackage[bookmarksnumbered, colorlinks, plainpages]{hyperref}

\usepackage{slashed}
\usepackage[capitalize]{cleveref}

\newtheorem{theorem}{Theorem}[section]
\newtheorem{lemma}[theorem]{Lemma}

\newtheorem{corollary}[theorem]{Corollary}

\theoremstyle{definition}
\newtheorem{definition}[theorem]{Definition}

\theoremstyle{remark}
\newtheorem{remark}[theorem]{Remark}
\numberwithin{equation}{section}

\newcommand*\diff{\mathop{}\!\mathrm{d}}

\DeclareMathOperator*{\esssup}{ess\,sup}

\crefname{equation}{}{}

\begin{document}
\setcounter{page}{1}

\title[Boundedness of operators on the torus. II]{Boundedness of pseudo-differential operators on the torus revisited. II}

\author[D. Cardona]{Duv\'an Cardona}
\address{
  Duv\'an Cardona:
  \endgraf
  Department of Mathematics: Analysis, Logic and Discrete Mathematics
  \endgraf
  Ghent University, Belgium
  \endgraf
  {\it E-mail address} {\rm duvanc306@gmail.com, duvan.cardonasanchez@ugent.be}
  }

\author[M. A. Mart\'inez]{Manuel Alejandro Mart\'inez}
\address{
  Manuel Alejandro Mart\'inez
  \endgraf
  Department of Mathematics
  \endgraf
 Universidad del Valle de Guatemala, Guatemala
  \endgraf
  {\it E-mail address} {\rm mar21403@uvg.edu.gt, manuelalejandromartinezf@gmail.com}
  }

\thanks{ Manuel Alejandro Mart\'inez has been supported by the {\it{Liderazgo en Ciencias}} scholarship of the Universidad del Valle de Guatemala.  Duv\'an Cardona is supported  by the FWO  Odysseus  1  grant  G.0H94.18N:  Analysis  and  Partial Differential Equations, by the Methusalem programme of the Ghent University Special Research Fund (BOF)
(Grant number 01M01021) and has been supported by the FWO Fellowship
Grant No 1204824N of the Belgian Research Foundation FWO}


\begin{abstract}
In this paper we continue our program of revisiting the new aspects about the boundedness properties of pseudo-differential operators on the torus. Here we prove $H^p$-$L^p$ and $H^p$-estimates for H\"ormander classes of pseudo-differential operators on the torus $\mathbb{T}^n$ for $p\leq 1$. The results are presented in the context of the global symbolic analysis developed by Ruzhansky and Turunen on $\mathbb{T}^n \times \mathbb{Z}^n$ by using the discrete Fourier analysis, which extends the $(\rho, \delta)$-H\"ormander classes on $\mathbb{T}^n$ defined by local coordinate systems. These results extend those proved by \'Alvarez and Hounie for the Euclidean case, considering even the case $\rho\leq\delta$.

\end{abstract} 

\maketitle
\tableofcontents
\allowdisplaybreaks

\section{Introduction}

The study of the boundedness of linear operators between Banach spaces is not only a classical problem in analysis, but useful in a variety of fields. For example, studying estimates of the norms of operators on $L^p$ spaces, Hardy spaces, Sobolev spaces, and other spaces of functions, in many situations, has been a crucial step when proving existence, analyzing the regularity and determining approximations of solutions of systems of linear and non-linear PDE's, see e.g. \cite{ruzhansky-turunen2}. In particular, Fourier and spectral multipliers, and, more generally, pseudo-differential operators provide a rich and accessible family of fundamental solutions for robust families of partial differential operators. Among them, one can mention the family of elliptic partial differential operators.

By continuing with our program of revisiting the new aspects about the boundedness properties of pseudo-differential operators on the torus, which has been started with the manuscript \cite{Cardona:Martinez} focused on Lebesgue spaces $L^p,$ $1<p<\infty,$ this work focused on the same problem for Hardy spaces $H^p,$ $0<p\leq 1,$  on the torus $\mathbb{T}^n:=\mathbb{R}^n /\mathbb{Z}^n$. 

On $\mathbb{R}^n$, the theory of Hardy spaces on several variables is thoroughly discussed by Fefferman and Stein in \cite{fefferman-stein}.  They proved that it is possible to apply the complex interpolation method between  $H^1$ and $L^2$, and $L^2$ and  $\mathrm{BMO}$, in order to obtain continuity properties concerning the Lebesgue spaces $L^p$. Moreover, Fefferman significantly discovered that the dual of the Hardy space $H^1$ is the space of functions with bounded mean oscillation $\mathrm{BMO}$, see \cite{fefferman-BMO}. These facts allowed Fefferman to prove the $L^p$-boundedness, $1<p<\infty,$ of pseudo-differential operators with symbols in the H\"ormander class $S^m_{\rho, \delta} (\mathbb{R}^n \times \mathbb{R}^n), $ where $0\leq\delta<\rho\leq1$ and $m\leq-n(1-\rho)|1/p-1/2|$. 
This result was extended to the torus by Delgado in \cite{delgado}, then to compact Lie groups by Delgado a Ruzhansky \cite{delgado-ruzhansky} and to manifolds with bounded geometry by G\'omez Cobos and Ruzhansky \cite{cobos-ruzhansky}. It has also been extended for subelliptic H\"ormander classes on compact Lie groups in \cite{cardona-ruzhansky-subelliptic}. For more details on the definition of pseudo-differential operators on the torus that allow for this extensions, we refer the reader to the foundational work of Agranovich \cite{agranovich} and McLane \cite{mclane}, and for more details on the calculus and quantization of toroidal pseudo-differential operators see Ruzhansky, Turunen and Vainikko \cite{ruzhansky-turunen, ruzhansky-turunen2, turunen-vainikko}. This subject will be revisited in \cref{section:prelims}.

On the other hand, \'Alvarez and Hounie \cite{alvarez-hounie} extended Fefferman's result on $\mathbb{R}^n$ to the case $\delta \geq \rho$ and, as a result, adapting the restriction of the order to $m\leq-n[(1-\rho)|1/p-1/2|+\lambda]$ where $\lambda=\max\{0, (\delta-\rho)/2 \}$. In \cite{Cardona:Martinez} we extend \'Alvarez and Hounie's $L^p$ continuity result to the case of the torus, considering the case $\delta\geq\rho$ as well. 
Continuing this program, in this paper we consider another result due to \'Alvarez and Hounie, now concerning the $H^p$-$L^p$ and the $H^p$ continuity properties of pseudo-differential operators defined on $\mathbb{R}^n$. We extend this result to the torus taking a similar approach as the one used in \cite{alvarez-hounie}. First, we consider the general case of operators with operator valued kernels and we prove the toroidal analogous of $H^p$-$L^p$ and $H^p$-boundedness results proved by \'Alvarez and Milman, see \cite{alvarez-milman}. In order to employ these results for the case of toroidal pseudo-differential operators, first we need to prove that their Schwartz kernels satisfy certain estimates when considered under an annulus "dyadic" decomposition. These estimates are analogous to the $D_{r,\alpha}$ conditions stated in \cite{alvarez-milman}, which are used to prove, for example, the $L^1$-weak-$L^1$ continuity of operators with operator valued kernel on $\mathbb{R}^n$. The toroidal counterpart of this theorem was proved in \cite{Cardona:Martinez}. These kernel estimates suffice to obtain the desired $H^p$-$L^p$-continuity of toroidal pseudo-differential operators. 
However, in order to prove the $H^p$-boundedness we define the concept of molecule, an object that behaves similar to an $H^p$-atom. Then, we prove that the image of a $(p,2)$-atom under a certain operator with operator valued kernel is one of such molecules, and their $H^p$-norms are uniform. Moreover, it is also required that the operator $T$ satisfies the $T^*(\overline{e})=0$ condition, where $\overline{e}$ is a constant function taking values in the dual space of a Banach space. This condition is similar to the $T(1)$ condition that David and Journ\'e showed to be a necessary and sufficient condition for the $L^2$-boundedness of Calder\'on-Zygmund integral operators \cite{david-journe}. This concludes in the $H^p$, $0<p\leq1$, boundedness of operators with operator valued kernel on the torus. As before, the kernel estimates of toroidal pseudo-differential operators allow us to employ this theorem for the case of interest of this paper. Now,
we proceed to state the main result of this paper, which is the toroidal counterpart of \'Alvarez and Hounie's result. Continuing forward, we will denote $\lambda = \max\{ 0, (\delta-\rho)/2 \}.$

\begin{theorem}
    Let $T\in \Psi^m_{\rho,\delta}(\mathbb{T}^n\times\mathbb{Z}^n)$, $0<\rho\leq1$, $0\leq\delta<1$. Assume that 
    \begin{equation*}
        m\leq-\beta-n\lambda \quad  \text{for some} \quad (1-\rho)\frac{n}{2}\leq\beta< \frac{n}{2}.
    \end{equation*}
    Then, the operator $T$ is a continuous mapping from $H^p(\mathbb{T}^n)$ into $L^p(\mathbb{T}^n)$ for $p$ such that,  $p_0\leq p\leq 1,$  when $\rho<1$, where 
    \begin{equation}
        \frac{1}{p_0} = \frac{1}{2} + \frac{\beta(1/\rho + n/2)}{n(1/\rho-1+\beta)},
        \label{eq:p0}
    \end{equation}
    and for $1\geq p > p_0=n/(n+1)$ when $\rho=1$. Moreover, if also
    $T^*(1)=0$ in the sense of BMO, then the operator $T$ is a continuous mapping from $H^p(\mathbb{T}^n)$ into itself for $p_0<p\leq1$, with $p_0$ as in \cref{eq:p0}.
\end{theorem}
Now, we proceed to discuss our main result.
\begin{remark}
    Notice that for the case $p=1$, this theorem recovers the $H^1(\mathbb{T}^n)$-$L^1(\mathbb{T}^n)$-boundedness proved in \cite{Cardona:Martinez}, which requires that the order of the toroidal pseudo-differential operator $T$ satisfies $m\leq -n[(1-\rho)/2 + \lambda]$. On the other hand, notice that the Hardy spaces $H^p$ are not stable under multiplication of test functions, thus this result cannot be obtained by extending \'Alvarez and Hounie's result using local partitions of unity. Moreover, $(\rho, \delta)$-H\"ormander symbol classes are not stable under change of variables whenever $\rho < 1-\delta$. Hence, we cannot use a partition of unity in the local coordinates of the torus $\mathbb{T}^n$, understood as a closed manifold, to obtain the desired results either. 
\end{remark}
This work is organized as follows. In \cref{section:prelims} we discuss the calculus of pseudo-differential operators, as developed by Ruzhansky and Turunen in \cite{ruzhansky-turunen}. Moreover, we define useful function spaces between which we will prove continuity, and we state conditions that will imply those properties. In \cref{section:Hp-Lp} we prove estimates for the Schwartz kernel of toroidal pseudo-differential operators. Also, we prove $H^p$-$L^p$-boundedness of operators with operator valued kernel, and we combine these properties with the kernel estimates of toroidal pseudo-differential operators to obtain the corresponding $H^p$-$L^p$ continuity of these operators for $p\leq1$. In \cref{section:Hp} we prove $H^p$ continuity for operators with operator valued kernel using molecules as auxiliary objects. Then we employ this result to prove the $H^p$-boundedness of toroidal pseudo-differential operators for $p\leq1$.

\section{Preliminaries}
\label{section:prelims}
In this section we present the preliminaries about the Fourier analysis on the torus and on the toroidal pseudo-differential calculus as developed in \cite{ruzhansky-turunen}.

\begin{definition}[Periodic functions] 
    For a Banach space $X$, a function $f: \mathbb{R}^n \rightarrow X$ is 1-periodic if $f(x) = f(x + k)$ for every $ x \in \mathbb{R}^n $ and $ k \in \mathbb{Z}^n $. We can identify this function with one defined in the torus $\mathbb{T}^n := \mathbb{R}^n / \mathbb{Z}^n $, which is a compact manifold without boundary. Moreover, the space of 1-periodic functions $m$ times continuously differentiable is denoted by $C^m(\mathbb{T}^n;X)$. The test functions are elements of the space of smooth functions $C^\infty(\mathbb{T}^n;X) := \bigcap_m C^m(\mathbb{T}^n;X)$. When $X=\mathbb{C}$, we  simply will use the notation $C^\infty(\mathbb{T}^n)$.
\end{definition}
In order to define the class of pseudo-differential operators on $\mathbb{T}^n$, we first need to define the Fourier transform for 1-periodic smooth functions.
\begin{definition}[Fourier Transform on $\mathbb{T}^n$]
    For $f \in C^\infty(\mathbb{T}^n)$ we define the  \textit{toroidal Fourier transform} $\mathcal{F}_{\mathbb{T}^n}$ as 
    \begin{equation*}
        (\mathcal{F}_{\mathbb{T}^n}f)(\xi) := \int_{\mathbb{T}^n} e^{-i2\pi x \cdot \xi}f(x)\diff x, 
    \end{equation*}
    for $\xi \in \mathbb{Z}^n$.
\end{definition}
In view of the representation theory of the torus $\mathbb{T}^n$, the corresponding frequency domain of 1-periodic smooth functions is the lattice $\mathbb{Z}^n$. Also, it can be proved that the corresponding functions $\mathcal{F}_{\mathbb{T}^n}f$ satisfy certain regularity conditions. In fact, they belong to the Schwartz space defined below. 
\begin{definition}[Schwartz space $\mathcal{S}(\mathbb{Z}^n) $]
    We say $\varphi \in \mathcal{S}(\mathbb{Z}^n)$, that is a \textit{rapidly decaying} function $\mathbb{Z}^n \rightarrow \mathbb{C}$, if for given $0<M<\infty$, there exists $C_{\varphi M} > 0$ such that 

    \begin{equation*}
        |\varphi(\xi)| \leq C_{\varphi M} \langle\xi\rangle^{-M} \quad , \quad \forall\, \xi \in \mathbb{Z}^n.
    \end{equation*}
\end{definition}
It can be proved that the toroidal Fourier transform is a bijection between $C^\infty(\mathbb{T}^n)$ and $\mathcal{S}(\mathbb{Z}^n)$ with inverse  $\mathcal{F}_{\mathbb{T}^n}^{-1}:  \mathcal{S}(\mathbb{Z}^n) \rightarrow C^\infty(\mathbb{T}^n)$ defined by
    \begin{equation*}
        (\mathcal{F}_{\mathbb{T}^n}^{-1}\varphi)(x) := \sum_{\xi \in \mathbb{Z}^n} e^{i2\pi x \cdot \xi} \varphi(\xi).
    \end{equation*}
In the Euclidean case, the Hörmander symbol classes are defined requiring some regularity conditions using partial derivatives in both, the space and frequency domain. However, in the case of the torus the frequency domain is the lattice $\mathbb{Z}^n$. Therefore, we define the natural analogous of the derivative in the discrete case: the difference operator.
\begin{definition}[Partial difference operator]
    Let $e_j \in \mathbb{Z}^n$ such that $(e_j)_i$ is equal to 1 if $i=j$ and 0 otherwise. Then, for $p(\xi):\mathbb{Z}^n \rightarrow \mathbb{C} $ we define 
    \begin{equation*}
        \Delta_{\xi_j}p(\xi) := p(\xi + e_j) - p(\xi) \quad \text{and} \quad \Delta^\alpha_\xi := \Delta_{\xi_1}^{\alpha_1} \cdots \Delta_{\xi_n}^{\alpha_n},
    \end{equation*}
    for any multi-index $\alpha \in \mathbb{N}^n_0$
\end{definition}
Now we proceed to define the Hörmander classes of symbols on $ \mathbb{T}^n \times \mathbb{Z}^n $, as in Ruzhansky and Turunen \cite{ruzhansky-turunen}, that will be important in the quantization of pseudo-differential operators.
\begin{definition}[Hörmander classes on $ \mathbb{T}^n \times \mathbb{Z}^n $]
    Let $m \in \mathbb{R} $ and $0 \leq \delta, \rho\leq 1$. We say that a function $p:=p(x, \xi)$ that is smooth on $x$ for any $\xi \in \mathbb{Z}^n$ belongs to the \textit{toroidal symbol class} $S^m_{\rho,\delta} (\mathbb{T}^n\times \mathbb{Z}^n) $ if 

    \begin{equation*}
        \left|\partial^\beta_x \Delta^\alpha_\xi p(x, \xi)\right| \leq C_{\alpha\beta}\langle\xi\rangle^{m-\rho|\alpha| + \delta|\beta|},
    \end{equation*}
    for every $x \in \mathbb{T}^n$, $\alpha, \beta \in \mathbb{N}^n_0$ and $\xi \in \mathbb{R}^n$. Moreover, we say that $p(x, \xi)$ has order $m$ and we define $S^{-\infty}_{\rho,\delta} (\mathbb{T}^n\times \mathbb{Z}^n) = \bigcap_{m\in\mathbb{R}}S^m_{\rho,\delta} (\mathbb{T}^n\times \mathbb{Z}^n) $.
\end{definition}

\begin{definition}[Pseudo-differential operators on $ \mathbb{T}^n \times \mathbb{Z}^n $]
    For a symbol $p:=p(x, \xi) \in S^m_{\rho,\delta} (\mathbb{T}^n\times \mathbb{Z}^n) $ we can associate a \textit{toroidal pseudo-differential operator} $\mathrm{Op}(p)$ from $C^\infty(\mathbb{T}^n)$ into itself, defined as  
    \begin{equation*}
        \mathrm{Op}(p)f(x):=\sum_{\xi \in \mathbb{Z}^n} e^{i2\pi x \cdot \xi}p(x, \xi)(\mathcal{F}_{\mathbb{T}^n}f)(\xi),
    \end{equation*}
    which can be rewritten as 
    \begin{equation}
        \mathrm{Op}(p)f(x)=  \sum_{\xi \in \mathbb{Z}^n} \int_{\mathbb{T}^n} e^{i2\pi (x-y) \cdot \xi} p(x, \xi)f(y) \diff y .
        \label{eq:pdo-def}
    \end{equation}
    The class of operators with symbols on $S^m_{\rho,\delta} (\mathbb{T}^n\times \mathbb{Z}^n) $ will be denoted  by $\Psi^m_{\rho,\delta} (\mathbb{T}^n\times \mathbb{Z}^n) $.
\end{definition}
    The toroidal Fourier transform and toroidal pseudo-differential operators can be extended by duality to the space of \textit{periodic distributions} $\mathcal{D}'(\mathbb{T}^n)$ consisting of continuous linear functionals on $C^\infty(\mathbb{T}^n)$. This extension allows us to define the Schwartz kernel of a toroidal pseudo-differential operator.
\begin{remark}[Schwartz kernel]
    Observe that the operator in \cref{eq:pdo-def} can be rewritten as
    \begin{equation*}
         \mathrm{Op}(p)f(x)=\int_{\mathbb{T}^n}\left[\sum_{\xi \in \mathbb{Z}^n} e^{i2\pi(x - y) \cdot \xi} p(x, \xi)\right]f(y)\diff y = \int_{\mathbb{T}^n}k(x, y)f(y)\diff y,
    \end{equation*}
    and we say $k(x, y)$ is the \textit{Schwartz kernel} of the corresponding operator. This is a well-defined distribution on $\mathbb{T}^n \times \mathbb{T}^n$. 
\end{remark}
Now we introduce a particular and very useful toroidal pseudo-differential operator.
\begin{definition}[Bessel's potential]
    We define the \textit{Bessel potential} $J^s$ as the pseudo-differential operator with symbol $\langle\xi\rangle^s$ for $s \in \mathbb{R}$.
\end{definition}

For our pourposes, we will use a smooth interpolation of a symbol on $\mathbb{T}^n \times \mathbb{Z}^n $ to obtain a symbol defined on $\mathbb{T}^n \times \mathbb{R}^n $. This is possible in view of the results stated in \cite[section~4.5]{ruzhansky-turunen}, here we write some useful consequences.

\begin{theorem}
    Let $m \in \mathbb{R}$, $0\leq\delta<1$, $0< \rho \leq 1$. The symbol $p \in S^m_{\rho,\delta} (\mathbb{T}^n\times \mathbb{Z}^n) $ is a toroidal symbol if and only if there exists a symbol $\tilde{p} \in S^m_{\rho,\delta} (\mathbb{T}^n\times \mathbb{R}^n) $ such that $p = \tilde{p}|_{\mathbb{T}^n\times \mathbb{Z}^n} $. Moreover, this extension is unique modulo $S^{-\infty}(\mathbb{T}^n\times \mathbb{R}^n)$.
    \label{theo:equivalence-symbols}
\end{theorem}

\begin{remark}
    When we employ this extension, the definition of its corresponding operator may be adjusted to 
    \begin{equation*}
        \mathrm{Op}(\tilde{p})f(x) = \int_{\mathbb{T}^n} \left[ \int_{\mathbb{R}^n} e^{i2\pi (x-y) \cdot \xi} \tilde{p}(x, \xi)\diff \xi \right] f(y) \diff y ,
    \end{equation*}
    so the Schwartz kernel of the operator, in the sense of distributions, is defined as the integral 
    \begin{equation*}
        \tilde{k}(x, y) = \int_{\mathbb{R}^n} e^{i2\pi (x-y) \cdot \xi} \tilde{p}(x, \xi)\diff \xi.
    \end{equation*}
    This kernel representation allows us to use techniques such as integration by parts when discussing properties of the Schwartz kernel.
\end{remark}
Hence, there is a correspondence between toroidal symbols with discrete and continuous frequency domains. This translates to the corresponding operators as well.
\begin{theorem}[Equivalence of operator classes]
    Let $m \in \mathbb{R}$, $0\leq\delta<1$, $0< \rho \leq 1$. Then 
        \begin{equation*}
            \Psi^m_{\rho,\delta}(\mathbb{T}^n\times \mathbb{R}^n) = \Psi^m_{\rho,\delta}(\mathbb{T}^n\times \mathbb{Z}^n).
        \end{equation*}
\end{theorem}
Now, we define some function spaces on the torus taking values on a Banach space $X$, used to prove our boudedness results.

\begin{definition}[Bochner-Lebesgue spaces]
    \cite[pp.~9]{defrancia} For a Banach space $X$ let us denote $L^p(\mathbb{T}^n;X)$, the \textit{Bochner-Lebesgue space}, consisting of strongly measureable $X$-valued functions $f$ defined on $\mathbb{T}^n$ such that 
    \begin{align*}
        \|f\|_{L^p(\mathbb{T}^n;X)} := \left( \int_{\mathbb{T}^n} \|f(x)\|_X^p \diff x \right)^{1/p} < \infty \; &, \quad 1\leq p < \infty,\\ 
        \|f\|_{L^\infty(\mathbb{T}^n;X)} := \esssup_{x \in \mathbb{T}^n} \|f(x)\|_X  < \infty. &
    \end{align*}
    When $X = \mathbb{C}$, we simply refer to these spaces as $L^p(\mathbb{T}^n).$ By simplifying the notation, we also use $\|\cdot\|_p$ for the $L^p$-norm.
    \label{def:bochner-lebesgue}
\end{definition}
We recall the definition of bounded linear operators between Banach spaces in order to define operators with operator valued kernel.
\begin{definition}[Bounded operators between Banach spaces $\mathcal{B}(X,Y)$]
    We say $T:X\rightarrow Y$ is a bounded linear operator between Banach spaces, namely $T \in \mathcal{B}(X, Y)$, if 
    \begin{equation*}
        \|Tx\|_Y \leq C\|x\|_X  \quad \text{for every } x\in X.
    \end{equation*}
    Also, the norm is defined as 
    \begin{equation*}
        \|T\|_{\mathcal{B}(X,Y)} := \inf \left\{ C \in \mathbb{R}^+ :  \|Tx\|_Y \leq C\|x\|_X \quad \text{for every } x\in X \right\}.
    \end{equation*}
    When $X$ and $Y$ are clear from context, we will simply use $\|\cdot\|_\mathcal{B}$.
\end{definition}
In order to prove continuity properties of a more general class of operators, we proceed to define operators with operator valued kernel. 
\begin{definition}[Operator valued kernel]
\cite[pp.~29]{defrancia} We say that an operator $T:C^\infty(\mathbb{T}^n;X)\rightarrow C^\infty(\mathbb{T}^n;Y)$ has an operator valued kernel if it can be written as 
\begin{equation*}
    Tf(x) = \int_{\mathbb{T}^n}k(x, y)f(y)\diff y,
\end{equation*}
 where $k:\mathbb{T}^n \times \mathbb{T}^n \rightarrow \mathcal{B}(X,Y)$, called the kernel, is such that $\|k(x, \cdot)\|_{\mathcal{B}(X,Y)}$ is integrable away of $x \in \mathbb{T}^n $.
    \label{def:operator-kernel}
\end{definition}
It has been shown that toroidal pseudo-differential operators in the Hörmander class of order zero are not bounded on $L^p(\mathbb{R}^n)$, in general for $p=1,\infty$, see e.g. the Ph.D. thesis of Wang \cite{wang}. This counterexample can be extended to the torus. Hence, it is required to define adequate subspaces from which we can interpolate to $L^p(\mathbb{T}^n)$ spaces for $1<p<\infty$. Those spaces happen to be the Hardy and BMO spaces, see \cite[pp.~159]{fefferman-stein}.
\begin{definition}[Hardy spaces $H^p(\mathbb{T}^n;X)$]
    For a Banach space $X$, we say that $a:\mathbb{T}^n\rightarrow X$ is an $H^{p, q}(\mathbb{T}^n;X)$\textit{-atom} if $a$ is supported on a ball $B \subset \mathbb{T}^n$ so that 
    \begin{equation*}
        \|a\|_{L^q(\mathbb{T}^n;X)} \leq |B|^{1/q-1/p},
    \end{equation*}
    and has zero moments up to $n(1/p -1)$. Also, we say that $f \in H^{p,q}(\mathbb{T}^n;X)$ if it can be written as  
    \begin{equation*}
        f = \sum_{j \in \mathbb{Z}^+}\lambda_ja_j \;\; \text{ where } \;\; \sum_{j \in \mathbb{Z}^+}|\lambda_j|^p < \infty,
    \end{equation*}
    and $a_j$ are $H^{p,q}(\mathbb{T}^n;X)$-atoms. An expression of this form is called a $(p, q)$-\textit{atomic decomposition of }$f$. Also, we define the $H^{p,q}(\mathbb{T}^n;X)$\textit{-norm} as 
    \begin{equation*}
        \|f\|_{H^{p,q}(\mathbb{T}^n;X)} := \inf \left(\sum_{j \in \mathbb{Z}^+}|\lambda_j|^p\right)^{1/p},
    \end{equation*}
    where the infimum is taken over all $(p,q)$-atomic decompositions of $f$. As in the scalar value case, $H^{p,q}(\mathbb{T}^n;X)=H^{p,r}(\mathbb{T}^n;X)$ even when $q\neq r$. Hence, we will simply denote the Hardy spaces by $H^{p}(\mathbb{T}^n;X)$, and when $X=\mathbb{C}$, by $H^{p}(\mathbb{T}^n)$
\end{definition}

\begin{definition}[Bounded Mean Oscillation space $\mathrm{BMO}(\mathbb{T}^n;X)$]
    We say $f \in \mathrm{BMO}(\mathbb{T}^n;X)$, namely, $f$ has \textit{bounded mean oscillation}, if 
    \begin{equation*}
        \|f\|_{\mathrm{BMO}(\mathbb{T}^n;X)}' := \sup_{B\subset\mathbb{T}^n} \frac{1}{|B|} \int_B \left\| f(x) - f_B \right\|_X \diff x < \infty, \quad \mathrm{ where } \quad f_B 
        := \frac{1}{|B|}\int_B f(x)\diff x,
    \end{equation*}
    and $B$ are balls contained on $\mathbb{T}^n$.
\end{definition}

In \cite{fefferman-BMO}, Charles Fefferman proved the famous result that $\mathrm{BMO}(\mathbb{R}^n)$ is the dual space of the Hardy space $H^1(\mathbb{R}^n)$. This result was extended to homogeneous groups, particularly the torus, see Folland and Stein \cite{folland-stein}. Moreover, if $X'$ satisfies the Radon-Nikodym property (in particular if $X$ is reflexive), then $\mathrm{BMO(\mathbb{T}^n;X')}$ is the dual of $H^1(\mathbb{T}^n;X)$. This duality can be understood on the following sense:
\begin{itemize}
    \item[(a)] For any $\phi \in \mathrm{BMO}(\mathbb{T}^n;X')$ we can define the functional $f\mapsto \int_{\mathbb{T}^n}f(x)\phi(x)\diff x$ which extends to a bounded functional on $H^1(\mathbb{T}^n;X)$.
    \item[(b)] Conversely, any continuous functional on $H^1(\mathbb{T}^n;X)$ can be identified with a functional defined as in (a) for a unique function $\phi \in \mathrm{BMO}(\mathbb{T}^n;X')$.
\end{itemize}
Now, we define maximal operators that would help us describe and obtain continuity properties for the operators in question.
\begin{definition}
    For a Banach space $X$ and any $f\in L^1(\mathbb{T}^n;X)$ we define the \textit{$p$-maximal function of Hardy-Littlewood} as 
    \begin{equation*}
        M_pf(x) = \sup_{B\ni x}\left(  \frac{1}{|B|}\int_B \|f(y)\|_X^p \diff y  \right)^{1/p}.
    \end{equation*}
    Moreover, we define the \textit{sharp maximal operator} as 
    \begin{equation*}
        f^\# (x) = \sup_{B\ni x} \inf_{c\in X} \left(  \frac{1}{|B|}\int_B \|f(y)-c\|_X\diff y  \right).
    \end{equation*}
\end{definition}
Notice that $\|f^\#\|_\infty$ is equivalent to $\|f\|_{\mathrm{BMO}}$. Therefore, we may employ this characterization when we consider it convenient. In our further analysis, we define the following "dyadic" decomposition in the annulus
    \begin{equation}
        A_j(z, \sigma) = \{ x \in \mathbb{T}^n : 2^j\sigma < |x-z|<2^{j+1}\sigma \} \; , \quad j=1, 2, 3, ...
    \end{equation}
\begin{remark}
    Note that $\mathbb{T}^n$ is contained in any ball with radius greater than $\sqrt{n}/2$. Hence, for any given $\sigma>0$, there is $N_\sigma\in \mathbb{Z}^+$ such that 
    \begin{equation}
        \frac{\sqrt{n}}{2\sigma} < 2^{N_\sigma} \leq \frac{\sqrt{n}}{\sigma}.
    \end{equation}
    Thus, this "dyadic" decomposition is finite in the case of the torus and $2^{N_\sigma} \sim \sigma^{-1}$.
    \label{rem:N-sigma}
\end{remark}
 We proceed to define a condition for operators with operator valued kernel that will be useful in proving continuity properties for these operators.
\begin{definition}[$D_{r,\alpha}$ condition] Let $1\leq r\leq\infty$ and $0<\alpha\leq1$. We say that an operator $T:C^\infty(\mathbb{T}^n;X)\rightarrow C^\infty(\mathbb{T}^n;Y)$ satisfies the $D_{r,\alpha}$ \textit{condition} if its associated operator valued kernel $k:=k(x,y)$ is continuous away from the diagonal of $\mathbb{T}^n\times\mathbb{T}^n$ and there exists a sequence $\{d_j\}\in \ell^1$ such that every $\sigma>0$ we have that 
\begin{equation*}
    \left( \int_{A_j(z,\sigma^\alpha)} \|k(x, y) - k(x, z)\|^r_\mathcal{B} \diff x \right)^{1/r} \leq d_j|A_j(z,\sigma^\alpha)|^{-1/r'}, \quad j=1,2,... ,
\end{equation*}
whenever $|y-z|<\sigma$. Moreover, we also require that $\tilde{k}(x, y) := k(y, x)$ satisfies these estimates.
    \label{def:D-r-alpha}
\end{definition}
A limiting case of this condition is stated as follows.
\begin{remark}[$D_\alpha$ condition]
    We say the operator $T$ from \cref{def:D-r-alpha} satisfies the $D_\alpha$ \textit{condition} if for some $0<\omega\leq1$, $0<\alpha\leq1$ we have that
    \begin{equation*}
        \|k(x,y)-k(x, z)\|_\mathcal{B}+\|k(y, x)-k(z, x)\|_\mathcal{B} \leq C \frac{|y-z|^\omega}{|x-z|^{n+\omega/\alpha}},
    \end{equation*}
    when $2|y-z|^\alpha \leq |x - z|$ for every $x, y, z\in \mathbb{T}^n$. 
    \label{rem:D-alpha-condition}
\end{remark}
Now, we define an object that has behavior similar to a $H^p$-atom. These functions will be helpful in proving continuity properties on $H^p$.
\begin{definition}[Molecule]
    For a Banach space $Y$, we say that $M:\mathbb{T}^n\rightarrow Y$ is a $(p, \theta, \mu)$\textit{-molecule} related to the ball $B(z, \sigma) \subset \mathbb{T}^n$ if it satisfies the following:
    \begin{itemize}
    \item If $\sigma \geq 1$:
    \begin{itemize}
        \item[($\mathbf{M_1}$)] \begin{equation*}
            \int \|M(x)\|_Y^2\diff x \leq C \sigma^{n(1-2/p)}.
        \end{equation*}
        \item[($\mathbf{M_2}$)] For some $2n/p-n < \mu < n + (2\omega/\alpha)$, we have that
        \begin{equation*}
            \int \|M(x)\|_Y^2|x-z|^\mu \diff x \leq C \sigma^{\mu + n(1-2/p)}.
        \end{equation*}
    \end{itemize}
    \item If $\sigma < 1$:
    \begin{itemize}
        \item[($\mathbf{M_1'}$)] \begin{equation*}
            \int \|M(x)\|_Y^2\diff x \leq C \sigma^{n(1/q-2/p)}.
        \end{equation*}
        \item[($\mathbf{M_2'}$)] For some $2n/p-n < \mu < (2\beta/(1-\theta)) \leq n + (2\omega/\alpha)$, we have that
        \begin{equation*}
            \int \|M(x)\|_Y^2|x-z|^\mu \diff x \leq C \sigma^{\theta\mu + n(1/q-2/p)}, 
        \end{equation*}
        where 
        \begin{equation*}
            \theta = \frac{n/2 + \omega - \beta}{n/2+\omega/\alpha}\leq \alpha \quad \text{ and } \quad \frac{1}{q} = \frac{1}{2} + \frac{\beta}{n}.
        \end{equation*}
    \end{itemize}
    \end{itemize}
    Moreover, it must satisfy the cancellation property, namely, that $\smallint M(x)\diff x = 0$.
    \label{def:molecule}
\end{definition}

\section{Main results}
\label{section:main}

In this paper, we set $\lambda := \max\{0, (\delta-\rho)/2\}$.
\subsection{Continuity of pseudo-differential operators from $H^p$ into $L^p$}
\label{section:Hp-Lp}
First, we consider the general case for operators with operator valued kernel. 
\begin{theorem}
\label{theo:operator-kernel-2j}
    Let $T$ be an operator with operator valued kernel $k:=k(x,y)$ satisfying for some $0<\omega\leq1$ the estimates
    \begin{equation}
        \int_{A_j(z, \sigma)} \| k(x, y) - k(x, z)\|_{\mathcal{B}(X, Y)} \diff x \leq C2^{-j\omega}, \quad \text{if} \quad \sigma\geq1;
    \end{equation}
    \begin{equation}
        \int_{A_j(z, \sigma^\gamma)} \| k(x, y) - k(x, z)\|_{\mathcal{B}(X,Y)} \diff x \leq C2^{-j\omega/\alpha}\sigma^{\omega(1-\gamma/\alpha)}, \quad \text{if} \quad \sigma<1,
    \end{equation}
    for $|y-z|<\sigma$ and any $0<\gamma\leq\alpha\leq1$. Moreover, suppose that the operator $T$ extends to a bounded operator from $L^2(\mathbb{T}^n; X)$ into $L^2(\mathbb{T}^n; Y)$ and from $L^q(\mathbb{T}^n; X)$ into $L^2(\mathbb{T}^n; Y)$, where 
    \begin{equation}
        \frac{1}{q} = \frac{1}{2} + \frac{\beta}{n}, \quad \text{for some} \quad (1-\alpha)\frac{n}{2} \leq \beta < \frac{n}{2}.
    \end{equation}
    Then, the operator $T$ is bounded from $H^p(\mathbb{T}^n; X)$ into $L^p(\mathbb{T}^n;Y)$, for $1\geq p \geq p_0$ when $\alpha<1$, where
    \begin{equation*}
             \frac{1}{p_0} = \frac{1}{2}+\frac{\beta(\omega/\alpha + n/2)}{n(\omega/\alpha-\omega+\beta)},
    \end{equation*}  
    and for $1\geq p > p_0=n/(n+\omega)$ when $\alpha=1$.
\end{theorem}

\begin{proof}
    Let us fix $0< p \leq1$ ant let $a$ be an $H^{p,\infty}(\mathbb{T}^n;X)$-atom supported on the ball $B(z,\sigma)$. First, let us assume that $\sigma\geq1$, then we employ the dyadic decomposition and the cancellation property of $a$ to obtain
    \begin{align*}
        \int_{\mathbb{T}^n}\|Ta(x)\|^p_Y \diff x \leq & \int_{B(z,\sigma)}\|Ta(x)\|^p_Y \diff x \\ 
        & + \sum_{j=1}^{N_\sigma} \int_{A_j(z,\sigma)}\left( \int_{B(z,\sigma)} \|k(x, y) - k(x, z)\|_\mathcal{B} \|a(y)\|_X 
        \diff y\right)^p \diff x \\
        = & I_1 + I_2
    \end{align*}
    Using H\"older's inequality with exponent $2/p$, and the $L^2$ continuity of $T$ we obtain
    \begin{align*}
        I_1 \leq & \left( \int_{\mathbb{T}^n} \|Ta(x)\|_Y^2\diff x \right)^{p/2}\left( \int_{\mathbb{T}^n} \chi_{B(z,\sigma)}(x)\diff x \right)^{(2-p)/2} \\
        = & \|Ta\|_{L^2(\mathbb{T}^n;Y)}^p|B|^{(2-p)/2} \\
        \leq & C\|a\|_{L^2(\mathbb{T}^n;Y)}^p |B|^{(2-p)/2} \\
        \leq & C |B|^{(p-2)/2} |B|^{(2-p)/2}\leq C.
    \end{align*}
    For $I_2$, we use H\"older's inequality with exponent $1/p$ and the kernel estimates of the hypothesis to get
    \begin{align*}
        I_2 \leq & \sum_{j=1}^{N_\sigma}\int_{A_j(z,\sigma)}\left( \int_{B(z,\sigma)} \|k(x, y) - k(x, z)\|_\mathcal{B} |B|^{-1/p} 
        \diff y\right)^p \diff x \\
        \leq & \sum_{j=1}^{N_\sigma} \left(\int_{B(z,\sigma)} \int_{A_j(z,\sigma)} \|k(x, y) - k(x, z)\|_\mathcal{B} \diff x\diff y \right)^p \left( \int_{A_j(z,\sigma)} |B|^{-1/(1-p)} \diff x \right)^{1-p} \\
        \leq &C \sum_{j=1}^{N_\sigma}  2^{-j\omega p}|B|^p \cdot |B|^{-1}|A_j(z, \sigma)|^{1-p} \\
        \leq &C \sum_{j=1}^{N_\sigma} 2^{-j\omega p}\sigma^{np} \cdot \sigma^{-n}2^{jn(1-p)}\sigma^{n(1-p)}\\
        = & C\sum_{j=1}^{N_\sigma} 2^{j[n-(n+\omega)p]},
    \end{align*}
    which can be bounded by a constant $C>0$ whenever 
    \begin{equation}
        p > \frac{n}{n+\omega} \geq \frac{n}{n+\omega/\alpha}.
        \label{eq:base-p0}
    \end{equation}
    Now, let us consider the case $\sigma<1$. Then 
    \begin{align*}
        \int_{\mathbb{T}^n}\|Ta(x)\|^p_Y \diff x \leq & \int_{B(z,2\sigma^\gamma)}\|Ta(x)\|^p_Y \diff x \\ 
        & + \sum_{j=1}^{N_\sigma} \int_{A_j(z,\sigma^\gamma)}\left( \int_{B(z,\sigma)} \|k(x, y) - k(x, z)\|_\mathcal{B} \|a(y)\|_X 
        \diff y\right)^p \diff x \\
        = & I_1 + I_2,
    \end{align*}
    with $\gamma$ to be chosen later.  The first term can be estimated using H\"older's inequality with exponent $2/p$ and the $L^q$-$L^2$ boundedness of $T$ to obtain
    \begin{align*}
        I_1\leq & \left( \int_{\mathbb{T}^n} \|Ta(x)\|_Y^2\diff x \right)^{p/2}\left( \int_{\mathbb{T}^n} \chi_{B(z,2\sigma^\gamma)}(x)\diff x \right)^{(2-p)/2} \\
        = & \|Ta\|_{L^2(\mathbb{T}^n;Y)}^p|B(z,2\sigma^\gamma)|^{(2-p)/2} \\
        \leq & C\|a\|_{L^q(\mathbb{T}^n;X)}^p\sigma^{n\gamma(2-p)/2} \\
        \leq & C\left( \int_{B(z, \sigma)}|B|^{-q/p} \right)^{p/q}\sigma^{n\gamma(2-p)/2}  \\
        \leq & C\sigma^{-n}\sigma^{np/q}\sigma^{n\gamma(2-p)/2} = C\sigma^{n[\gamma(1-p/2) +p/q-1]}.
    \end{align*}
    Thus, since $1/q=1/2+\beta/n$ we can conclude that $I_1$ will be bounded whenever 
    \begin{equation}
        \gamma \geq \frac{2n - p(n + 2\beta)}{n(2-p)}, 
        \label{eq:gamma-lowe-bound}
    \end{equation}
    which is an decreasing function of $p$ and reaches $\gamma\geq1-2\beta/n$ when $p=1$. Hence the requirement $\beta \geq (1-\alpha) n/2$.
    On the other hand, using H\"older's inequality with exponent $1/p$ and the kernel estimate from the hypothesis, we obtain 
    \begin{align*}
        I_2 \leq & \sum_{j=1}^{N_{\sigma^\gamma}}\int_{A_j(z,2\sigma^\gamma)}\left( \int_{B(z,\sigma)} \|k(x, y) - k(x, z)\|_\mathcal{B} |B|^{-1/p} 
        \diff y\right)^p \diff x \\
        \leq & \sum_{j=1}^{N_{\sigma^\gamma}} \left(\int_{B(z,\sigma)} \int_{A_j(z,2\sigma^\gamma)} \|k(x, y) - k(x, z)\|_\mathcal{B} \diff x\diff y \right)^p \left( \int_{A_j(z,2\sigma^\gamma)} |B|^{-1/(1-p)} \diff x \right)^{1-p} \\
        \leq & \sum_{j=1}^{N_{\sigma^\gamma}}  C2^{-j\omega p/\alpha}\sigma^{\omega p(1-\gamma/\alpha)}|B|^p \cdot |B|^{-1}|A_j(z, 2\sigma^\gamma)|^{1-p} \\
        \leq & \sum_{j=1}^{N_{\sigma^\gamma}} C2^{-j\omega p/\alpha}\sigma^{\omega p(1-\gamma/\alpha)}\sigma^{np} \cdot \sigma^{-n}2^{jn(1-p)}2^{n(1-p)}\sigma^{n\gamma(1-p)}\\
        = & C\sigma^{-\gamma[p(n+\omega/\alpha)-n] +p(n+\omega)-n} \sum_{j=1}^{N_{\sigma^\gamma}} 2^{j[n-(n+\omega/\alpha)p]}
    \end{align*}
    Now, since $p$ must satisfy \cref{eq:base-p0}, we have that $n - (n+\omega/\alpha)p < 0$. Hence, we can bound the geometric sum by a constant and obtain that
    \begin{align*}
        I_2 \leq & C\sigma^{-\gamma[p(n+\omega/\alpha)-n] +p(n+\omega)-n},
    \end{align*}
    which can be estimated by a constant whenever 
    \begin{equation}
        \gamma \leq \frac{p(n+\omega)-n}{p(n+\omega/\alpha)-n},
        \label{eq:gamma-upper-bound}
    \end{equation}
    which is an increasing function of $p$ that reaches $\gamma \leq \alpha$ when $p=1$. Hence, the critical $p_0$ occurs when we equal the right hand sides of \cref{eq:gamma-lowe-bound} and \cref{eq:gamma-upper-bound}, completing the proof.
\end{proof} 

Now, we prove that under certain conditions, toroidal pseudo-differential operators satisfy the kernel estimates needed to use the previous theorem. 
\begin{theorem}
    Let $T\in \Psi^m_{\rho,\delta}(\mathbb{T}^n\times\mathbb{Z}^n)$, $0<\rho\leq1$, $0\leq\delta<1$ with kernel $k:=k(x, y)$. Then, 
    \begin{itemize}
        \item[a)] If $\sigma \geq \varepsilon>0$, and $j=1,2,3,...$,
        \begin{equation}
            \sup_{|y-z|<\sigma}\int_{A_j(z, \sigma)} |k(y, x) - k(z, x)|\diff x \leq C_\varepsilon2^{-j},
        \end{equation}
        \begin{equation}
            \sup_{|y-z|<\sigma}\int_{A_j(z, \sigma)} |k(x, y) - k(x, z)|\diff x \leq C_\varepsilon2^{-j},
        \end{equation}
        where $C_\varepsilon$ does not depend on $\sigma$, $j$, or $z$.
        \item[b)] If $m\leq -n[(1-\rho)/2 +\lambda]$, $0<\gamma\leq1$, $\sigma<1$, and $j=1,2,3,...$, 
        \begin{equation}
            \sup_{|y-z|<\sigma}\int_{A_j(z, \sigma^\gamma)} |k(x, y) - k(x, z)|\diff x \leq C2^{-j/\rho}\sigma^{1-\gamma/\rho}.
        \end{equation}
        
        \item[c)] If $m\leq -n(1-\rho)/2 $, $0<\gamma\leq1$, $\sigma<1$, and $j=1,2,3,...$, 
        \begin{equation}
            \sup_{|y-z|<\sigma}\int_{A_j(z, \sigma^\gamma)} |k(y, x) - k(z, x)|\diff x \leq C2^{-j/\rho}\sigma^{1-\gamma/\rho}.
            \label{eq:kernel-estimate-c}
        \end{equation}
        
    \end{itemize}
    \label{theo:pdo-kernel}
\end{theorem}
\begin{proof}
    \begin{itemize}
        \item[a)] By \cite[Theorem~3.1]{Cardona:Martinez} and the triangle inequality we have that 
        \begin{align*}
            \int_{A_j(z, \sigma)} |k(y, x) - k(z, x)|\diff x \leq &  \int_{A_j(z, \sigma)} |k(y, x)| \diff x +\int_{A_j(z, \sigma)} |k(z, x)|\diff x\\
            \leq & C\int_{A_j(z, \sigma)} |x-y|^{-N} \diff x + C\int_{A_j(z, \sigma)} |x-z|^{-N} \diff x,
        \end{align*}
        for some $N\geq (m+n)/\rho$. Now, we have that $|x-y| \geq |x-z| - |z-y|>2^j\sigma -\sigma \geq 2^{j-1}\sigma$. Thus, since the torus has volume one, we obtain that 
        \begin{align*}
            \int_{A_j(z, \sigma)} |k(y, x) - k(z, x)|\diff x \leq & C(2^{j-1}\sigma)^{-N}  +  C(2^{j}\sigma)^{-N}  \\
            \leq & C(2^{j} \varepsilon)^{-N} \leq C_\varepsilon 2^{-j}.
        \end{align*}

        \item[b)] Let $\tilde{p}:=\tilde{p}(x,\xi)$ be the corresponding symbol of $T$ defined on $\mathbb{T}^n \times \mathbb{R}^n$, see \cref{theo:equivalence-symbols}.  Let $\varphi \in C_0^\infty(\mathbb{R})$ be supported in $[1/2, 1]$ such that  
        \begin{equation*}
            \int_0^\infty \varphi(1/t)/t\diff t = \int_1^2 \varphi(1/t)/t\diff t=1.
        \end{equation*}
        Define 
        \begin{equation*}
            k(x, y, t) = \int_{\mathbb{R}^n} e^{i2\pi (x-y) \cdot \xi} \tilde{p}(x, \xi) \varphi(\langle\xi\rangle/t)\diff t,
        \end{equation*}
        so that 
        \begin{equation*}
            k(x, y) = \int_0^\infty k(x, y, t) \diff t = \int_1^\infty k(x, y, t) \diff t.
        \end{equation*}
        For $0<\gamma\leq1$, we have that 
        \begin{equation*}
            \int_{A_j(z, \sigma^\gamma)} |k(x, y, t) - k(x, z, t)| \diff x \leq
        \end{equation*}
        \begin{equation*}
             \left[ \int_{\mathbb{T}^n} \left(1 + t^{2\rho} |x-z|^2\right)^N 
            |k(x, y, t) - k(x, z, t)|^2 \diff x
            \right]^{1/2}  \left[ \int_{A_j(z, \sigma^\gamma)} \left(1 + t^{2\rho} |x-z|^2\right)^{-N} \diff x  \right]^{1/2},
        \end{equation*}
        where $N>n/2$ is a natural number to be determined. In \cite[Theorem~3.4]{Cardona:Martinez} it is proved that the left hand side is dominated by 
        \begin{equation*}
            C\sigma t \cdot t^{\rho n/2} \quad \text{if} \quad \sigma t\leq 1. 
        \end{equation*}
        From here comes the order restriction. To estimate the second factor, let us define 
        \begin{equation*}
            F(r) = \left[ \int_r^{2r} (1+s^2)^{-N} s^{n-1} \diff s  \right]^{1/2}, \quad 0<r<\infty. 
        \end{equation*}
         Note that $F$ is a smooth function, such that $F(r) \sim r^{n/2}$ as $r\rightarrow 0$ and $F(r)\sim r^{n/2-N}$ as $r \rightarrow \infty$. Hence, we obtain that
        \begin{equation*}
            \left[ \int_{A_j(z, \sigma^\gamma)} \left(1 + t^{2\rho} |x-z|^2\right)^{-N}  \right]^{1/2}
        \end{equation*}
        \begin{equation*}
            \leq \left[ \int_{A_j(z, \sigma^\gamma)} \left(1 + t^{2\rho} |x-z|^2\right)^{-N} (t^\rho|x-z|)^{n-1} (t^\rho 2^j \sigma^\gamma)^{1-n} \diff x  \right]^{1/2}
        \end{equation*}
        \begin{equation*}
            \leq C t^{-\rho n/2} F(t^\rho2^j\sigma^\gamma).
        \end{equation*}
        Thus, we have that 
        \begin{equation*}
            \int_{A_j(z, \sigma^\gamma)} |k(x, y, t) - k(x, z, t)| \diff x \leq Ct\sigma F(t^\rho 2^j\sigma^\gamma), \quad t\sigma \leq 1.
        \end{equation*}
        Now, let us consider the case $t\sigma>1$. The computation made in \cite[Theorem~3.4]{Cardona:Martinez}, shows that 
        \begin{equation*}
            \int_{A_j(z, \sigma^\gamma)} |k(x, y, t)| + |k(x, z, t)| \diff x \leq C (t^\rho 2^j \sigma^\gamma )^{n/2 - N}.
        \end{equation*} 
        Combining, the last two estimates we obtain that 
        \begin{equation*}
            I_{j}(y, z, t) := \int_{A_j(z, \sigma^\gamma)} |k(x, y, t)-k(x, z, t)| \diff x 
        \end{equation*}
        \begin{equation}
            \leq 
            C \left[  \int_1^{1/\sigma} t\sigma F(t^\rho 2^j \sigma^\gamma )/t\diff t + \int_{1/\sigma}^\infty (t^\rho 2^j \sigma^\gamma )^{n/2 - N}/t \diff t \right].
            \label{eq:kernel-estimates-sum}
        \end{equation}
        Now, we choose $N$ so that $\rho(N-n/2)>1$, which implies that $\smallint F(t^\rho)\diff t < \infty$. Moreover, from \cref{eq:kernel-estimates-sum} we obtain that the integral $I_j(y, z, t)$ can be estimated by
        \begin{equation*}
            C\left(  2^{-j/\rho} \sigma^{1-\gamma/\rho} + 2^{j(n/2 -N)} \sigma^{(1-\gamma/\rho)\rho(N-n/2)} 
            \right)  \leq C2^{-j/\rho} \sigma^{1-\gamma/\rho}, \quad 0<\sigma<1.
        \end{equation*}
        Thus, completing the proof of this case. 

        \item[c)] We can use the same method as in the previous case to estimate \cref{eq:kernel-estimate-c} when $t\sigma\leq1$. Moreover, by inspecting the proof of \cite[Theorem~3.4]{Cardona:Martinez} we can see that we can estimate it as follows 
        \begin{equation*}
            \int_{A_j(z, \sigma^\gamma)} |k(y, x) - k(z, x)|\diff x  
        \end{equation*}
        \begin{equation*}
            \leq C \left[  \int_1^{1/\sigma} t\sigma F(t^\rho 2^j \sigma^\gamma )/t\diff t + \int_{1/\sigma}^\infty (t^\rho2^j\sigma^\gamma)^{n/2 - N} \diff t   \right].
        \end{equation*}        
    \end{itemize}
    Obtaining the desired result.
\end{proof}
We proceed to use these estimates and \cref{theo:operator-kernel-2j} to obtain the $H^p$-$L^p$ boundedness for toroidal pseudo-differential operators. 
\begin{theorem}
    Let $T\in \Psi^m_{\rho,\delta}(\mathbb{T}^n\times\mathbb{Z}^n)$, $0<\rho\leq1$, $0\leq\delta<1$. Assume that 
    \begin{equation*}
        m\leq-\beta-n\lambda \quad  \text{for some} \quad (1-\rho)\frac{n}{2}\leq\beta< \frac{n}{2}.
    \end{equation*}
    Then, the operator $T$ is a continuous mapping from $H^p(\mathbb{T}^n)$ into $L^p(\mathbb{T}^n)$ for $1 \geq p \geq p_0$ when $\rho<1$, where 
    \begin{equation}
        \frac{1}{p_0} = \frac{1}{2} + \frac{\beta(1/\rho + n/2)}{n(1/\rho-1+\beta)},
    \end{equation}
    and for $1\geq p > p_0=n/(n+1)$ when $\rho=1$.
    \label{theo:Hp-Lp}
\end{theorem}

\begin{proof}
    Notice that by \cref{theo:pdo-kernel}, the operator $T$ satisfies the kernel estimates with $\alpha=\rho$ and $\omega=1$. The $L^2$-boundedness was proved in \cite[Theorem~3.6]{Cardona:Martinez} since $m\leq -n\lambda$. Thus, the only requirement from \cref{theo:operator-kernel-2j} that we need to prove is the $L^q$-$L^2$-boundedness. Notice that $J^{-\beta}$ is bounded from $L^q(\mathbb{T}^n)$ into $L^2(\mathbb{T}^n)$ by the Hardy-Littlewood-Sobolev inequality. Hence, since $J^\beta T$ has order $m+\beta \leq -n\lambda$, we obtain that 
    \begin{equation}
    \|Tf\|_q =    \|J^{-\beta}(J^\beta T)f\|_q \leq C \|(J^\beta T)f\|_2 \leq C\|f\|_2.
    \label{eq:q-2-boundedness}
    \end{equation}
    Thus, completing the requirements to obtain the desired result from \cref{theo:operator-kernel-2j}.
\end{proof}
\begin{remark}
    Notice that for $p=1$ and $\beta=n(1-\rho)/2$ we obtain the $H^1$-$L^1$ continuity proved in \cite[Theorem~3.10]{Cardona:Martinez}
\end{remark}
\subsection{Continuity of pseudo-differential operators on $H^p$}
\label{section:Hp}

We will prove that the image of a $(p, 2)$-atom is a suitable molecule, see \cref{def:molecule}. Therefore, in order to make sense of the cancellation property of a molecule, we prove the following lemma.

\begin{lemma}
    Any $(p,\theta,\mu)$-molecule $M:=M(x)$ related to a ball $B(z, \sigma)$ is an absolutely integrable function.
\end{lemma}
\begin{proof}
    First, let us assume that $\sigma\geq1$. Then, using H\"older's inequality and ($\mathbf{M}_1$) we have that
    \begin{align*}
        \int_{B(z,\sigma)}\|M(x)\|_Y \diff x \leq &  \|M(x)\|_{L^2(\mathbb{T}^n;Y)} \|\chi_{B(z,\sigma)}\|_2 \\
        \leq & C\sigma^{n(1/2-1/p)}\cdot \sigma^{n/2} \leq C|B|^{1-1/p}.
    \end{align*}
    On the other hand, by H\"older's inequality and ($\mathbf{M}_2$)
    \begin{align*}
        \int_{\mathbb{T}^n\setminus B(z,\sigma)}\|M(x)\|_Y\diff x \leq &  \left\|M(x)|x-z|^{\mu/2} \right\|_{L^2(\mathbb{T}^n;Y)} \left\| |x-z|^{-\mu/2} \chi_{\mathbb{T}^n\setminus B(z,\sigma)}(x) \right\|_2 \\
        \leq & C \sigma^{\mu/2 + n(1/2-1/p)} \cdot\sigma^{(n-\mu)/2} \\
        \leq & C \sigma^{n(1-1/p)} \leq C |B|^{1-1/p}.
    \end{align*}
    Hence, $M(x) \in L^1$ when $\sigma \geq 1$. Now, let us assume that $\sigma < 1$. Using ($\mathbf{M}_1'$) and H\"older's inequality we obtain that
    \begin{align*}
        \int_{B(z,\sigma)}\|M(x)\|_Y \diff x \leq & 
        \|M(x)\|_{L^2(\mathbb{T}^n;Y)}\|\chi_{B(z,\sigma)}\|_2 \\
        \leq & C \sigma^{n(1/q - 1/p)} \cdot \sigma^{n/2}\\ \leq & C|B|^{1/q-1/p+1/2} = C|B|^{\beta/n+1-1/p},
    \end{align*}
    where we used the fact that $1/q=1/2+\beta/n$. Using ($\mathbf{M}_2'$) we get that
    \begin{align*}
        \int_{\mathbb{T}^n\setminus B(z,\sigma)}\|M(x)\|_Y\diff x \leq &  \left\|M(x)|x-z|^{\mu/2} \right\|_{L^2(\mathbb{T}^n;Y)} \left\| |x-z|^{-\mu/2} \chi_{\mathbb{T}^n\setminus B(z,\sigma)}(x) \right\|_2 \\
        \leq & C\sigma^{\theta\mu/2+n(1/q-1/p)} \cdot\sigma^{(n-\mu)/2} \\
        = & C|B|^{\beta/n+1-1/p-(1-\theta)\mu/2n}.
    \end{align*}
    Thus, $M(x) \in L^1$ also when $\sigma<1$, completing the proof.
\end{proof}
Now, we prove that the $H^p$-norm of a molecule only depends on the constants related to it, see \cref{def:molecule}. This will be helpful when proving that the $H^p$-norm of the image of a $(p, 2)$-atom under a certain operator is uniform.
\begin{lemma}
    Let $M:=M(x)$ be a $(p, \theta, \mu)$-molecule related to $B(z, \sigma)$. Then, 
    \begin{equation*}
        M(x)=\sum_{j=0}^{N_\sigma} \lambda_j a_j,
    \end{equation*}
    where $a_j$ is a $H^{p,2}(\mathbb{T}^n;Y)$-atom supported on $B(z, 2^{j+1}\sigma)$. Moreover $\|M\|_{H^p(\mathbb{T}^n; Y)}$ only depends on the constants $C$ in \cref{def:molecule}.
    \label{lem:M-in-Hp}
\end{lemma}

\begin{proof}
    Let $B_j = B(z, 2^{j+1}\sigma)$ and let $M_j$ be the mean value of $M(x)$ in $A_j(z, \sigma)$. Notice that $A_j = B_j - B_{j-1}$. Let us define 
    \begin{equation*}
        \psi_j(x) = [M(x)-M_j]\chi_{A_j}(x),
    \end{equation*}
    which is supported on $B_j$ and has mean value zero. Moreover, 
    \begin{equation*}
        \|\psi_j\|_{L^2(\mathbb{T}^n;Y)}^2 \leq 2^2 \int_{A_j(z,\sigma)}\|M(x)\|_Y^2 \diff x.
    \end{equation*}
    Now, let us assume $\sigma\geq1$ and use ($\mathbf{M}_2$) to obtain
    \begin{align*}
        \|\psi_j\|_{L^2(\mathbb{T}^n;Y)}^2 \leq & C \int_{A_j(z,\sigma)}\|M(x)\|_Y^2|x-z|^{\mu} |x-z|^{-\mu} \diff x
        \\
        \leq & C\sigma^{\mu+n(1-2/p)}\cdot2^{-j\mu} \sigma^{-\mu}
        \\
        \leq & C 2^{-j[\mu + n(1-2/p)]} |B_j|^{1-2/p}.
    \end{align*}
    Hence, we have that 
    \begin{equation}
        \|\psi_j\|_{L^2(\mathbb{T}^n;Y)} \leq C 2^{-j[\mu/2 + n(1/2-1/p)]} |B_j|^{1/2-1/p}.
        \label{eq:psi-j}
    \end{equation}
    On the other hand, we have that
    \begin{equation*}
        M(x)\chi_{B_m}(x) - \sum_{j=0}^m\psi_j(x) = \sum_{j=0}^mM_j\chi_{B_j}(x),
    \end{equation*}
    where the left hand side converges in the $L^2$-norm to $M(x) -\sum_{j=0}^\infty\psi_j(x)$. To estimate the left-hand side, we define a sequence $\{\nu_j\}$, $j=-1, 0, 1, 2,...,N_\sigma$ by
    \begin{equation*}
        \nu_{-1} = \int M(x)\diff x = 0, \quad \nu_j = \int_{\mathbb{T}^n\setminus B_j} M(x)\diff x.
    \end{equation*}
    Then we have that
    \begin{equation*}
        \sum_{j=0}^{N_\sigma}M_j \chi_{A_j} =\sum_{j=0}^{N_\sigma}(\nu_{j-1}-\nu_j)|A_j|^{-1}\chi_{A_j} = \sum_{j=0}^{N_\sigma-1}\phi_j - \nu_{N_\sigma}|A_{N_\sigma}|^{-1}\chi_{A_{N_\sigma}},
    \end{equation*}
    where $\phi_j=\nu_j\left( |A_{j+1}|^{-1}\chi_{A_{j+1}} - |A_{j}|^{-1}\chi_{A_{j}}  \right)$. We can notice that each $\phi_j$ is supported on $B_{j+1}$ and has mean value zero. Furthermore, we can use the method used for $\psi_j$ to obtain
    \begin{equation}
        \|\phi_j\|_{L^2(\mathbb{T}^n;Y)} \leq C 2^{-j[\mu/2 + n(1/2-1/p)]} |B_{j+1}|^{1/2-1/p}.
        \label{eq:phi-j}
    \end{equation}
    Now, we have that $\nu_{N_\sigma}=0$ thanks to the fact that $B_{N_\sigma}$ contains the torus, see \cref{rem:N-sigma}. Hence we have that 
    \begin{equation}
        M(x)= \psi_0(x)+ \sum_{j=1}^{N_\sigma} \left[\psi_j(x) + \phi_{j-1}(x) \right],
        \label{eq:psi-phi}
    \end{equation}
    where each term can be rewritten as a $H^{p,2}(\mathbb{T}^n; Y) $-atom by \cref{eq:psi-j} and \cref{eq:phi-j}. Now, let us assume $\sigma<1$ and use ($\mathbf{M}_2'$) to obtain
    \begin{align*}
        \|\psi_j\|_{L^2(\mathbb{T}^n;Y)}^2 \leq & C \int_{A_j(z,\sigma)}\|M(x)\|_Y^2|x-z|^{\mu} |x-z|^{-\mu} \diff x
        \\
        \leq & C\sigma^{\theta\mu+2n(1/q-1/p)}\cdot2^{-j\mu} \sigma^{-\mu}
        \\
        \leq & C 2^{-j[\mu + 2n(1/2-1/p)]} \sigma^{\mu(\theta-1) + 2n(1/q-1/2)} |B_j|^{2(1/2-1/p)}.
    \end{align*}
    Hence, we have that 
    \begin{align*}
        \|\psi_j\|_{L^2(\mathbb{T}^n;Y)} \leq C 2^{-j[\mu/2 + n(1/2-1/p)]} \sigma^{\mu(\theta-1)/2+n(1/q-1/2)} |B_j|^{1/2-1/p}.
    \end{align*}
    
    The exponent of $\sigma$ becomes $\mu(\theta-1)/2 + \beta$, hence the restriction $\mu \leq 2\beta/(1-\theta)$. On the other hand, $\|\phi_j\|_2$ can be estimated in a similar way to obtain the expression \cref{eq:psi-phi} for $\sigma<1$. Now, we can consider the $H^p$-norm of $M$. In this case, we have that
    \begin{equation*}
        |\lambda_j|^p \leq C 2^{j[ n(1-p/2) -\mu p/2 ]},
    \end{equation*}
    Hence $\|M\|_{H^p(\mathbb{T}^n;Y)} \leq C$ whenever $\mu > np/2 -n$. Thus, completing the proof.
\end{proof}
Now, we prove an auxiliary result that relates the $p$-maximal function of Hardy-Littlewood and the sharp maximal operator.
\begin{theorem}
    Let $T$ be an operator satisfying the $D_{r,\alpha}$ condition such that for $1<p<q\leq\infty$ and $p/q\leq\alpha$ we have that
    \begin{equation}
        \left( \frac{1}{|B(z, \sigma)|} \int_{B(z, \sigma)} \|Tf(x)\|_Y^q \diff x  \right)^{1/q} \leq C \left( \frac{1}{|B(z, \sigma^\alpha)|} \int \|f(x)\|_X^p \diff x  \right)^{1/p}, 
        \label{eq:c-p-q-alpha}
    \end{equation}
    for every $0<\sigma<1$ and 
    \begin{equation*}
        \|Tf\|_{L^p(\mathbb{T}^n;Y)} \leq C \|f\|_{L^p(\mathbb{T}^n;X)},
    \end{equation*}
    when $\sigma\geq1$, for some absolute constant $C>0$. Then, for $s=\max \{ p, r' \}$, we have that 
    \begin{equation}
        (Tf)^\#(x) \leq CM_sf(x), \quad f\in L^\infty (\mathbb{T}^n; X).
        \label{eq:Tf-Msf}
    \end{equation}
    \label{theo:Tf-Msf}
\end{theorem}

\begin{proof}
    Let us fix a ball $B(z, \sigma)$ and let us write $f=f_1+f_2$, where 
    \begin{equation*}
        f_1 = f\chi_{B(z, 2\sigma^\alpha)}.
    \end{equation*}
    Let
    \begin{equation*}
        c = \int k(z, y)f_2(y)\diff y,
    \end{equation*}
    then we have that
    \begin{align*}
        \int_{B(z, \sigma)} \|Tf(x)-c\|_Y \diff x \leq & \int_{B(z, \sigma)} \|Tf_1(x)\|_Y \diff x \\
        & + \sum_{j=1}^{N_{\sigma^\alpha}} \int_{B(z, \sigma)}\int_{A_j(z, \sigma^\alpha)} \|k(x,y)-k(z,y)\|_\mathcal{B}\|f(y)\|_X \diff y \diff x .
    \end{align*}
    First, using H\"older's inequality and the $D_{r, \alpha}$ condition we obtain that for $j=1,2,...,N_{\sigma^\alpha}$
    \begin{equation*}
    \int_{A_j(z, \sigma^\alpha)} \|k(x,y)-k(z,y)\|_\mathcal{B}\|f(y)\|_X \diff y 
    \end{equation*}
    \begin{equation*}
        \leq\left( \int_{A_j(z, \sigma^\alpha)} \|k(x,y)-k(z,y)\|_\mathcal{B}^r \diff y \right)^{1/r}\left(  \int_{A_j(z,\sigma^\alpha)}  \|f(y)\|_X^{r'} \diff y \right)^{1/r'}
    \end{equation*}
    \begin{equation*}
        \leq  d_j |A_j(z,\sigma^\alpha)|^{-1/r'}\left(  \int_{A_j(z,\sigma^\alpha)}  \|f(y)\|_X^{r'} \diff y \right)^{1/r'}.
    \end{equation*}
    \begin{equation*}
        \leq d_j M_{r'}f(z).
    \end{equation*}
    Hence, we have that 
    \begin{equation*}
        \frac{1}{|B(z,\sigma)|}\sum_{j=1}^{N_{\sigma^\alpha}} \int_{B(z,\sigma)}\int_{A_j(z, \sigma^\alpha)} \|k(x,y)-k(z,y)\|_\mathcal{B}\|f(y)\|_X \diff y \diff x
    \end{equation*}
    \begin{equation*}
        \leq CM_{r'}f(z).
    \end{equation*}
    On the other hand, let us assume $\sigma<1$. Then, using H\"older's inequality and \cref{eq:c-p-q-alpha}, we get that
    \begin{equation*}
        \frac{1}{|B(z,\sigma)|}\int_{B(z, \sigma)}\|Tf_1(x)\|_Y \diff x 
    \end{equation*}
    \begin{equation*}
        \leq \left( 
        \frac{1}{|B(z,\sigma)|}\int_{B(z,\sigma)} \|Tf_1(x)\|^q_Y \diff x 
        \right)^{1/q} \left( \int_{B(z,\sigma)} \frac{\diff x}{|B(z,\sigma)|}  
        \right)^{1/q'}
    \end{equation*}
    \begin{equation*}
        \leq C \left( \frac{1}{|B(z, \sigma^\alpha)|} \int_{B(z,2\sigma^\alpha)} \|f(x)\|_X^p \diff x  \right)^{1/p} 
    \end{equation*}
    \begin{equation*}
        \leq C M_pf(z).
    \end{equation*}
    For the case $\sigma \geq 1$ we use H\"older's inequality and the $L^p$-boundedness of $T$ to obtain
    \begin{align*}
        \frac{1}{|B(z,\sigma)|}\int_{B(z, \sigma)}\|Tf_1(x)\|_Y \diff x  \leq & 
        |B(z,\sigma)|^{-1/p} \|Tf_1\|_{L^p(\mathbb{T}^n;Y)} \\
        \leq & C |B(z,\sigma)|^{-1/p} \|f_1\|_{L^p(\mathbb{T}^n;X)} \\
        \leq & C M_pf(z).
    \end{align*}
    Combining the estimates above we can conclude that 
    \begin{equation*}
        (Tf)^\#(z) \leq C M_sf(z),
    \end{equation*}
    finishing the proof.
\end{proof}
Moreover, when the operator satisfies the $D_{1,\alpha}$ condition we obtain the following corollary.
\begin{corollary}
    Let $T$ be an operator as in \cref{theo:Tf-Msf}, but that satisfies the $D_{1,\alpha}$ condition. Then $T$ is a continuous operator from $L^\infty(\mathbb{T}^n;X)$ into $\mathrm{BMO}(\mathbb{T}^n;Y)$.
    \label{cor:L-inf-BMO}
\end{corollary}

\begin{proof}
    Notice that $r=1$ implies that $s=\infty$ and the estimate in \cref{eq:Tf-Msf} becomes
    \begin{equation*}
        (Tf)^\#(x) \leq C\|f\|_{L^\infty(\mathbb{T}^n;X)},
    \end{equation*}
    proving the result.
\end{proof}
In the hypothesis of the following lemma we have the $T^*(\overline{e})$ condition, a vector valued counterpart of the $T(1)$ condition famously stated by David and Journ\'e, see \cite{david-journe}.
\begin{lemma}
    Let $X, Y$ be reflexive Banach spaces, let $T$ and  $T^*$ be operators with kernels satisfying the $D_\alpha$ condition. Also, assume that they can be extended to bounded operators from $L^2(\mathbb{T}^n;X)$ into $L^2(\mathbb{T}^n;Y)$ and from $L^q(\mathbb{T}^n;X)$ into $L^2(\mathbb{T}^n;Y)$ so that 
    \begin{equation}
        \frac{1}{q} = \frac{1}{2} + \frac{\beta}{n} \quad \text{ for some } \quad (1-\alpha)\frac{n}{2}\leq \beta<\frac{n}{2}.
    \end{equation}
    Moreover, suppose that $T^*(\overline{e})=0$ for every ${e}\in Y'$ and $\overline{e}:\mathbb{T}^n\rightarrow Y'$ given by $\overline{e}(x)=e$. Let $a:=a(x)$ be an $H^{p,2}(\mathbb{T}^n;X)$-atom supported on $B(z,\sigma)$. Then, $M(x):=Ta(x)$ is a $(p, \theta, \mu)$-molecule with constants $C$ depending only on $T$ and its continuity properties.
    \label{lem:Ta-is-molecule}
\end{lemma}

\begin{proof}
    First, let us assume $\sigma \geq 1$. Then, by the $L^2$ boundedness we get that
    \begin{equation*}
        \int\|M(x)\|_Y^2\diff x = \|Ta\|_{L^2(\mathbb{T}^n;Y)}^2 \leq C \|a\|_{L^2(\mathbb{T}^n;Y)}^2 \leq C |B|^{2(1/2-1/p)} \leq C\sigma^{n(1-2/p)},
    \end{equation*}
    proving $(\mathbf{M_1})$. Moreover, we have that 
    \begin{align*}
        \int \|M(x)\|_Y^2|x-z|^\mu \diff x \leq &
        \int_{B(z, 2\sigma^\alpha)} \|M(x)\|_Y^2|x-z|^\mu \diff x\\
        & + \sum_{j=1}^{N_{\sigma^\alpha}} \int_{A_j(z, \sigma^\alpha)} \|M(x)\|_Y^2|x-z|^\mu \diff x \\
        = & I_1 + I_2.
    \end{align*}
    By $(\mathbf{M_1})$, we obtain that
    \begin{equation*}
        I_1 \leq C\int\|M(x)\|_Y^2 \sigma^{\alpha\mu} \diff x\leq  C \sigma^{\mu +n(1-2/p)}.
    \end{equation*}
    Moreover, assuming $x\in A_j(z, \sigma^\alpha)$, and by the cancellation property of $a$ we get that
    \begin{align*}
        \|M(x)\|_Y^2 \leq & \left[ \int_{B(z, \sigma)} \|k(x, y) - k(x,z)\|_\mathcal{B} \|a(y)\|_X \diff y \right]^2 \\
        \leq & C\frac{|y-z|^{2\omega}}{|x-z|^{2(n+\omega/\alpha)}} 
        \left[ \int_{B(z, \sigma)} \|a(y)\|_X \diff y \right]^2 \\
        \leq & C \frac{\sigma^{2\omega} \cdot \sigma^{2n(1-1/p)}}{|x-z|^{2(n+\omega/\alpha)}}.
    \end{align*}
    Now, since $\mu < n + 2\omega/\alpha$ we have that
    \begin{align*}
        I_2 \leq & \sum_{j=1}^{N_{\sigma^\alpha}} \sigma^{2\omega + 2n(1-1/p)} \cdot (2^j\sigma^\alpha)^{\mu - 2\omega/\alpha - n} \\
        \leq &  C \sigma^{\alpha\mu + 2n(1-1/p) - n\alpha} \\ 
        \leq & C \sigma^{\mu + 2n(1-1/p)},
    \end{align*}
    since $\mu > n$ and $\alpha \leq 1$, proving $(\mathbf{M_2})$. Now, let us suppose that $\sigma < 1$. First, by the $L^q$-$L^2$-boundedness of $T$, we get that 
    \begin{equation*}
        \int\|M(x)\|_Y^2\diff x = \|Ta\|_{L^2(\mathbb{T}^n;Y)}^2 \leq C \|a\|_{L^q(\mathbb{T}^n;Y)}^2 \leq C |B|^{2(1/q-1/p)} \leq C\sigma^{2n(1/q-1/p)},
    \end{equation*}
    proving $(\mathbf{M_1'})$. On the other hand, we have
    \begin{align*}
        \int \|M(x)\|_Y^2|x-z|^\mu \diff x \leq &
        \int_{B(z, 2\sigma^\theta)} \|M(x)\|_Y^2|x-z|^\mu \diff x\\
        & + \sum_{j=1}^{N_{\sigma^\theta}} \int_{A_j(z, \sigma^\theta)} \|M(x)\|_Y^2|x-z|^\mu \diff x \\
        = & I_1 + I_2.
    \end{align*}
    As in the case above, we can use $(\mathbf{M_1'})$ to estimate 
    \begin{equation*}
        I_1 \leq C \sigma^{\theta\mu + 2n(1/q - 1/p)}.
    \end{equation*}
    Moreover, when $x \in A_j(z, \sigma^\theta)$ we have an estimate similar as above, namely that
    \begin{align*}
        I_2 \leq & C \sum_{j=1}^{N_{\sigma^\theta}} \sigma^{2\omega + 2n(1 - 1/p) } (2^j\sigma^\theta)^{\mu - n - 2\omega/\alpha} \\
        \leq & C \sigma^{ 2\omega + 2n(1-1/p) + \theta\mu -\theta(n + 2\omega/\alpha) } \\
        = & C\sigma^{ \theta\mu + 2n(1/q - 1/p) },
    \end{align*}
    since 
    \begin{equation*}
        \theta = \frac{n(1-1/q)+\omega}{ n/2 + \omega/\alpha }.
    \end{equation*}
    This completes the proof of $(\mathbf{M_2'})$. It only remains to prove that $\int M(x)\diff x=0$. Notice that $|B|^{-1+1/p}a$  is a $H^{1,2}(\mathbb{T}^n;X)$-atom and given $e \in Y'$ we have that
    \begin{align*}
        |B|^{-1+1/p}\left\langle e, \int M(x)\diff x \right\rangle =&  \int \left\langle \overline{e}(x) , T\left(  |B|^{-1+1/p}a\right)(x)\right\rangle \diff x \\
        = & \int \left\langle T^*\overline{e}(x) ,   |B|^{-1+1/p}a(x)\right\rangle \diff x = 0.
    \end{align*}
    Where we have used the fact that $T^*$ is continuous on $L^2$ to employ \cref{cor:L-inf-BMO} and conclude that $T^*$ maps $L^\infty(\mathbb{T}^n, Y')$ into $\mathrm{BMO}(\mathbb{T}^n; X')$, which is the dual of $H^1(\mathbb{T}^n;X)$, since $X$ is reflexive. Hence, we proved that $M(x)$ is a $(p, \theta,\mu)$-molecule.
\end{proof}
Notice that the related constants of the resulting molecule only depend on the operator $T$. Moreover, the $H^p$-norm of the molecule only depend on such constants. Hence, we proceed to combine the previous results of this section to obtain the $H^p$-boundedness for operators with operator valued kernels under certain conditions. 
\begin{theorem}
    Let $X, Y$ be reflexive Banach spaces, let $T$ and  $T^*$ be operators satisfying the $D_\alpha$ condition. Also, assume that they can be extended to bounded operators from $L^2(\mathbb{T}^n;X)$ into $L^2(\mathbb{T}^n;Y)$ and from $L^q(\mathbb{T}^n;X)$ into $L^2(\mathbb{T}^n;Y)$ so that 
    \begin{equation}
        \frac{1}{q} = \frac{1}{2} + \frac{\beta}{n} \quad \text{ for some } \quad (1-\alpha)\frac{n}{2}\leq \beta<\frac{n}{2}.
    \end{equation}
    Moreover, suppose that $T^*(\overline{e})=0$ for every ${e}\in Y'$ and $\overline{e}:\mathbb{T}^n\rightarrow Y'$ given by $\overline{e}(x)=e$. Let 
    \begin{equation}
        \frac{1}{p_0}=\frac{1}{2} + \frac{\beta(\omega/\alpha+n/2)}{n(\omega/\alpha - \omega +\beta)},
    \end{equation}
    then for $p_0<p\leq1$, the operator $T$ is bounded from $H^p(\mathbb{T}^n;X)$ into $H^p(\mathbb{T}^n;Y)$.
    \label{theo:Hp-operator}
\end{theorem}

\begin{proof}
    Notice that by \cref{lem:Ta-is-molecule}, the image of every $H^{p,2}(\mathbb{T}^n;X)$-atom is a $(p, \theta,\mu)$-molecule with constants $C$ depending only on $T$, and by \cref{lem:M-in-Hp} we have that such molecules have $H^p$-norms depending only in said constants. Hence we have a uniform bound $\|Ta\|_{H^p(\mathbb{T}^n;Y)}\leq C$ for every $(p,2)$-atom $a$. Inspecting \cref{def:molecule} we can see that this happens when 
    \begin{equation*}
        2n/p -n < \frac{2\beta}{1-\theta}, \quad \text{ where } \quad \theta = \frac{n/2+\omega-\beta}{n/2+\omega/\alpha}.
    \end{equation*}
    Thus, we can conclude the expression for $p_0$ stated above.
\end{proof}
Now, we apply this result in the specific context of toroidal pseudo-differential operators.
\begin{theorem}
    Let $T\in \Psi^m_{\rho,\delta}(\mathbb{T}^n\times\mathbb{Z}^n)$, $0<\rho\leq1$, $0\leq\delta<1$. Assume that
    \begin{equation}
        m\leq-\beta-n\lambda \quad  \text{for some} \quad (1-\rho)\frac{n}{2}\leq\beta\leq \frac{n}{2}, 
    \end{equation}
     and that $T^*(1)=0$ in the sense of BMO. Then the operator $T$ is a continuous mapping from $H^p(\mathbb{T}^n)$ into itself for $p_0<p\leq1$ where 
     \begin{equation}
         \frac{1}{p_0} = \frac{1}{2}+\frac{\beta(1/\rho + n/2)}{n(1/\rho-1+\beta)}.
     \end{equation}
\end{theorem}
\begin{proof}
    We can choose $\gamma =\rho$ in \cref{theo:pdo-kernel} in order to obtain the $D_{1,\rho}$ condition for $T$ and $T^*$. Hence these operators satisfy the $D_\rho$ condition  with $\omega=1$ (see \cref{rem:D-alpha-condition}). Moreover, $T$ is bounded on $L^2$ as proved in \cite[Theorem~3.6]{Cardona:Martinez} and the $L^q$-$L^2$-boundedness of $T$ follows from \cref{eq:q-2-boundedness} in the proof \cref{theo:Hp-Lp}. Thus we have proved that $T$ satisfies all the conditions to apply \cref{theo:Hp-operator}. The proof of the theorem is complete.
\end{proof}

\bibliographystyle{amsplain}

\end{document}